\documentclass{amsart}
\usepackage[applemac]{inputenc}
\usepackage{amsmath}
\usepackage{amssymb}
\usepackage{bbm}
\usepackage{tikz-cd}
\usepackage{latexsym}
\usepackage{xcolor}

\numberwithin{equation}{section}
\usepackage{tikz}
\usetikzlibrary{arrows}
\usetikzlibrary{patterns}
\usetikzlibrary{shapes}
\usepackage{mathrsfs}
\usepackage{bbm}
\usepackage{enumerate}
\usepackage{lmodern}
\usepackage{mathtools}

\newcommand{\CP}{\mathbb{C}\mathrm{P}}
\newcommand{\C}{\mathbb{C}}

 \usepackage[usenames,dvipsnames]{pstricks}
 \usepackage{epsfig}
 \usepackage{pst-grad} 
 \usepackage{pst-plot} 

\newtheorem{thm}{Theorem}[section]

\newtheorem{cor}[thm]{Corollary}
\newtheorem{prop}[thm]{Proposition}

\theoremstyle{remark} 
\newtheorem{remark}{Remark}[section]

\newcommand{\be}{\begin{equation}}
\newcommand{\ee}{\end{equation}}

\usepackage{mathtools} 
\mathtoolsset{showonlyrefs,showmanualtags} 

\ifx\pdfpageheight\undefined
   \usepackage[dvips,colorlinks=true,linkcolor=blue,citecolor=blue,%
      urlcolor=purple]{hyperref}
   \usepackage[dvips]{graphicx}
   \makeatletter
   \edef\Gin@extensions{\Gin@extensions,.mps}
   \makeatother
\else
   \usepackage[bookmarksopen=false,pdftex=true,breaklinks=true,%
      backref=page,pagebackref=true,plainpages=false,%
      hyperindex=true,pdfstartview=FitH,colorlinks=true,%
      pdfpagelabels=true,colorlinks=true,linkcolor=blue,%
      citecolor=blue,urlcolor=green,hypertexnames=false%
      ]%
   {hyperref}
\fi
\usepackage{bm}

\title[Symplectic instability of B\'ezout's theorem]{Symplectic instability of B\'ezout's theorem}
\author{Michele Ancona and Antonio Lerario}
\begin{document}

\maketitle
\begin{abstract}We investigate the failure of B\'ezout's Theorem for two symplectic surfaces in $\CP^2$ (and more generally on an algebraic surface), by proving that every plane algebraic curve $C$ can be perturbed in the $\mathscr{C}^\infty$--topology to an arbitrarily close smooth symplectic surface $C_\epsilon$ with the property that the cardinality $\#C_\epsilon \cap Z_d$ of the transversal intersection of $C_\epsilon$ with an algebraic plane curve $Z_d$ of degree $d$, as a function of $d$ can grow arbitrarily fast. As a consequence we obtain that, although B\'ezout's Theorem is true for pseudoholomorphic curves with respect to the same almost complex structure, it is ``arbitrarly false" for pseudoholomorphic curves with respect to different (but arbitrarily close) almost--complex structures (we call this phenomenon ``instability of B\'ezout's Theorem").\end{abstract}

\section{Introduction}
B\'ezout's Theorem says that if two algebraic curves $C_1$ and $C_2$ in $\CP^2$ intersect transversally, then they meet exactly in $\deg(C_1)\cdot \deg(C_2)$ points. Pseudoholomorphic curves also share this property. More precisely, let $J$ be an almost--complex structure on $\CP^2$ and $C_1$ and $C_2$ be two $J$--holomorphic curves of degree $d_1$ and $d_2$ (meaning that the homology classes $[C_1]$ and $[C_2]$ equals, respectively $d_1[\CP^1]$ and $d_2[\CP^1]$). Then, if $C_1$ and $C_2$ intersect transversally, the number of points in $C_1\cap C_2$ is exactly equal to $d_1\cdot d_2$. 

Pesudoholomorphic curves are a fundamental tool in symplectic geometry, since the work of Gromov \cite{gromovpseudo}. When the almost--complex structure $J$ is tamed by a symplectic form $\omega$ (meaning $\omega(v,Jv)>0$, for any tangent vector $v\neq 0$), then a $J$--holomorphic curve is a symplectic surface. Conversely, given a symplectic surface $C\subset \CP^2$, there exists a tamed almost--complex structure $J$ for which $C$ is a $J$--holomorphic curve \cite{mc}. It is then natural to investigate whether some form of B\'ezout's Theorem is still valid for symplectic surfaces. 

Similar homological bounds are not valid in general for the intersection of two symplectic surfaces in $\CP^2$ without the assumption that both surfaces are pseudoholomorphic for the same almost--complex structure $J$. In this paper we will prove, in a strong sense, that one cannot expect \emph{any} type of bound at all (this will be made clear below). We will say that in the symplectic framework B\'ezout's Theorem is  ``arbitrarly false" (see Remark \ref{rmk:symplectic}). This will be a consequence of the following theorem.

\begin{thm}\label{thm:CP1}Let $C\hookrightarrow \CP^2$ be a smooth algebraic curve and $\{a_d\}_{d\in \mathbb{N}}$ be a sequence of positive natural numbers. For every $\epsilon>0$ there exists a $\mathscr{C}^{\infty}$--surface $C_\epsilon\hookrightarrow \CP^2$, $\epsilon$--close to $C$ in the $\mathscr{C}^\infty$--topology, and a sequence $\{Z_d\}_{d\in \mathbb{N}}$ of smooth algebraic curves, with $d=\deg(Z_d)$, such that for infinitely many $d\in \mathbb{N}$:
\be\label{eq:in} \#Z_d\cap C_\epsilon \geq a_d.\ee
Moreover the sequence $\{Z_d\}_d\in \mathbb{N}$ can be chosen such that the intersection $Z_d\cap C_\epsilon$ is transversal for every $d\in \mathbb{N}.$
\end{thm}
We collect now some comments and consequences of Theorem \ref{thm:CP1}. 

\begin{remark}\label{rmk:symplectic} The standard complex structure of $\CP^2$ is tamed by the standard Fubini--Study symplectic form $\omega_{\mathrm{FS}}$ and in particular the algebraic curves are symplectic surfaces.  ``Being symplectic" is an open condition for the $\mathscr{C}^1$--topology and so a small $\mathscr{C}^\infty$--perturbation  of a symplectic surface remains a symplectic surface. Then, the $\mathscr{C}^{\infty}$--surface $C_\epsilon$ given by Theorem \ref{thm:CP1} can be chosen to be symplectic. In particular, Theorem \ref{thm:CP1} says that B\'ezout is  ``arbitrarily false'' for symplectic surfaces.
\end{remark}
\begin{remark}
The perturbation $C_{\epsilon}$ of $C$ is constructed using a small parameter $\epsilon$. Varying this parameter yields a smooth family of $\mathscr{C^\infty}$--surfaces that tend to $C$ in the $\mathscr{C}^\infty$--topology when $\epsilon\rightarrow 0$. As explained in Remark \ref{rmk:symplectic}, for $\epsilon$ small enough, $C_\epsilon$ is a symplectic surface and hence it is a $J_\epsilon$--holomorphic curve for an almost--complex structure $J_\epsilon$ tamed by $\omega_{\mathrm{FS}}$. We then have $J_\epsilon\rightarrow J_0$ in the $\mathscr{C}^\infty$--topology, where $J_0$ is the standard complex structure of $\CP^2$. Theorem \ref{thm:CP1} shows that, although for two $J$--holomorphic curves B\'ezout's Theorem is true,  this is ``arbitrarily false" for two pseudoholomorphic curves with respect to two arbitrarily close almost--complex structures. We call this phenomenon ``instability of B\'ezout theorem", from which the title of the article.
\end{remark}
\begin{remark}\label{rmk:Smooth is not needed} As explained earlier, our goal is to study B\'ezout's Theorem for symplectic surfaces in $\CP^2$. In order to construct symplectic surfaces that do not respect B\'ezout we start from smooth plane algebraic curves and apply a small smooth perturbation, as explained in Remark \ref{rmk:symplectic}. However, as will be seen in the proof, the construction of the perturbation $C_\epsilon$ of $C$ is local and it is done on small disks on $C$. The global smoothness of the algebraic curve $C$ is indeed not necessary and Theorem \ref{thm:CP1} is still valid if we start from a singular algebraic curve $C$: it is enough to apply the procedure described in Section \ref{Section Proof theorem 1} around a smooth point of $C$. 
\end{remark}
\begin{remark}\label{rmk:subanalytic}As it will become clear from the proof, for every $k\geq 0$ we can also require that the surfaces $C_\epsilon$ are of class $\mathscr{C}^k$ and \emph{subanalytic} (this is due to the existence of $\mathscr{C}^k$ and subanalytic bump functions, see Remark \ref{remark:sub}). In particular, no quantitative estimate can be expected even in the tame world.  For a ``real'' counterpart of this problem, see \cite{BLN, GKP}.
\end{remark}

Another way to look at B\'ezout's Theorem is from the point of view of positivity of intersection: two plane algebraic curves (or two $J$--holomorphic curves) always intersect positively and then, when the intersection is transversal, the number of intersection points coincides with the intersection product in homology, which is the product of the degrees of the curves. This point of view allows us to extend this problem on any algebraic surface $X$: for two algebraic curves in $X$ with transversal intersection, the number of intersection points coincides with the intersection product in homology. This positivity property remains valid for $J$--holomorphic curves in $X$, if we fix an almost--complex structure $J$ on $X$. In this sense, Theorem \ref{thm:CP1} is a special case of the following theorem.

\begin{thm}\label{thm:surfaces}Let $X$ be an algebraic surface, $H$ be a very ample divisor, $C\hookrightarrow X$ be an algebraic curve and $\{a_d\}_{d\in \mathbb{N}}$ be a sequence of positive numbers. For every $\epsilon>0$ there exists a $\mathscr{C}^\infty$--surface $C_\epsilon\hookrightarrow X$, $\epsilon$--close to $C$ in the $\mathscr{C}^\infty$--topology, and a sequence of algebraic curves $\{Z_d\}_{d\in \mathbb{N}}$, with $[Z_d]=d[H]$, such that for infinitely many $d\in \mathbb{N}$:
\be \#Z_d\cap C_\epsilon \geq a_d.\ee
Moreover the sequence $\{Z_d\}_d\in \mathbb{N}$ can be chosen such that the intersection $Z_d\cap C_\epsilon$ is transversal for every $d\in \mathbb{N}.$
\end{thm}
It is worth noting that all previous Remarks \ref{rmk:symplectic}--\ref{rmk:subanalytic}  remain valid for any algebraic surface (not just for $\CP^2$) and can be applied verbatim to Theorem \ref{thm:surfaces}.

\subsection{Structure of the paper}In Section \ref{Section Proof theorem 1} we prove Theorem \ref{thm:CP1} in the case of $\CP^1\hookrightarrow \CP^2$.  The proof of the general case of Theorem \ref{thm:CP1} is given in Section \ref{Section Proof theorem 1 general}. Finally in Section \ref{sec:proj} we prove Theorem \ref{thm:surfaces}. 
\subsection*{Acknowledgements}The authors thank the anonymous referee for her/his helpful comments, which simplified the proof of the results.

\section{Proof of Theorem \ref{thm:CP1}: the case $C=\CP^1$}\label{Section Proof theorem 1}
For the sake of exposition, we first give the proof of Theorem \ref{thm:CP1} in the special case $C=\CP^1,$ and then extend the construction to include the case of an algebraic curve $C\hookrightarrow \CP^2$. The perturbation $C_\epsilon$ will be the image of an embedding $\eta_\epsilon:\CP^1\to \CP^2$ which on $\CP^1\setminus D$ equals the inclusion (here $D$ will be a disk), and on $D$ it will be a small perturbation of the inclusion.

Let us start by fixing some notation. We denote by $[z_0, z_1, z_2]$ the homogeneous coordinates on $\CP^2$. In these coordinates we have $\CP^1=\{z_2=0\}$. We denote by $H:=\{z_0=0\}$ the hyperplane at infinity and by $\phi:\CP^2\setminus H\to \C^2$ the analytic affine chart
\be \phi:[z_0, z_1, z_2]\mapsto \left(\frac{z_1}{z_0}, \frac{z_2}{z_0}\right).\ee
We have $\phi(\CP^1\setminus H)=\phi(\CP^1\setminus [0,1,0])=\{(w_1, w_2)\in \C^2\,|\, w_2=0\}$ and $\phi([1, 0, 0])=(0,0).$

Let $B_r\subset\C$ be the disk of radius $r>0$ centered at the origin. Denote by $D_r:=\phi^{-1}(B_{r}\times \{0\})$ (this is a neighborhood of $[1,0,0]$ in $\CP^1$), so that $\phi$ induces a biholomorphism of triples:

\be ( \CP^2\setminus H, \CP^1\setminus [0,1,0],D_r)\simeq (\C^2, \C\times \{0\}, B_{r}).\ee
For $r=1$ we will just write $B$ for $B_1$ and $D$ for $D_1$.
The proof of Theorem \ref{thm:CP1} relies on the following key proposition.
\begin{prop}\label{propo:disk}Let $D:=\phi^{-1}(B\times \{0\})\subset \CP^1\subset \CP^2$. Let $j_{0}:D\to \CP^2$ be the inclusion. For every sequence $\{a_d\}_{d\geq 1}$ of natural numbers there exists a smooth embedding $j_\epsilon:D\to \CP^2$ with the property that for every $\epsilon>0$
\begin{enumerate}
\item the embedding $j_{\epsilon}|_{D\setminus D_{\frac{2}{3}}}$ equals $j_{0}|_{D\setminus D_{\frac{2}{3}}}$;
\item  the embedding $j_{ \epsilon}|_{D_\frac{1}{2}}$ is holomorphic;
\item as $\epsilon\to 0$ the embedding $j_{\epsilon}\to j_{0}$ in the $\mathscr{C}^\infty$--topology;
\end{enumerate}
Moreover there exists a subsequence $\{a_{d_m}\}_{m\geq 1}$ and, for every $\epsilon>0$, a sequence of plane algebraic curves $\{\widetilde{Z}_{d_m}\}_{m\geq1}$ with $\deg(\widetilde{Z}_{d_m})=d_m$, such that:
\begin{enumerate}
\item [(4)] for every $m\geq 1$ the intersection $j_{\epsilon}(D)\cap \widetilde{Z}_{d_m}$ is transversal;
\item [(5)] for every $m\geq 1$ the intersection $j_{\epsilon}(D_{\frac{1}{2}})\cap \widetilde{Z}_{d_m}$ consists of $a_{d_m}$ points, all of which are positively oriented.
\end{enumerate}
\end{prop}
We will prove Proposition \ref{propo:disk} in the Section \ref{sec:ppd}, we now show how Theorem \ref{thm:CP1}, in the case $C=\CP^1$, follows from it.
\begin{proof}[Proof of Theorem \ref{thm:CP1}: case $C=\CP^1$]Using the notation of Proposition \ref{propo:disk}, define $C_\epsilon$ to be 
\be C_\epsilon:=(\CP^1\setminus D_{\frac{2}{3}})\cup j_{\epsilon}(D),\ee
and consider the sequence $\{\widetilde{Z}_{d_m}\}_{m\geq 1}$ of algebraic curves given by the same proposition

For every $\epsilon>0$ the set $C_\epsilon$ is a well defined smooth surface in $\CP^2$ by point (1) of Proposition \ref{propo:disk} and by point (5) we can make it arbitrarily close to $\CP^1$ in the $\mathscr{C}^\infty$--topology. Point (3) and (4) of the proposition give now for every $m\geq 1$:
\be  \#\widetilde{Z}_{d_m}\cap C_\epsilon \geq a_{d_m},\ee
with transversal intersection in $j_{\epsilon}(D)$. The curves $\{Z_{d_m}\}_{m\geq 1}$ are defined now by taking a small perturbation of the polynomials defining the $\{\widetilde{Z}_{d_m}\}_{m\geq 1}$, making them smooth and with the intersection which is transversal on the whole $C_\epsilon$. This concludes the proof. \end{proof}

\begin{cor}With the above notations, if $\epsilon>0$ is small enough then $C_{\epsilon}\hookrightarrow \CP^2$ is symplectic.
\end{cor}
\begin{proof}This is immediate, since  $\CP^1\hookrightarrow \CP^2$ is symplectic and ``being symplectic" is an open property in the $\mathscr{C}^1$--topology.
\end{proof}
\subsection{Proof of Proposition \ref{propo:disk}}\label{sec:ppd}
\subsubsection{Definition of the embedding $j_{\epsilon}$.}\label{sec:prtb1}
The definition of the embedding $j_{\epsilon}$ depends on the choice of a smooth map $\varphi:B\to \C$, which we explicitly construct in Section \ref{sec:map}. Once the map $\varphi$ has been fixed, the embedding $j_{\epsilon}$ will be defined as follows. Let $\rho:\C\to [0,1]$ be a smooth bump function such that
\be\label{eq:bump1} \rho|_{\C\setminus B_{\frac{2}{3}}}\equiv 0, \quad 0<\rho|_{B_{\frac{2}{3}}\setminus B_{\frac{1}{2}}}<1\quad \textrm{and}\quad \rho|_{B_\frac{1}{2}}\equiv 1.\ee
\begin{remark}\label{remark:sub}If in this step we take $\rho$ to be a $\mathscr{C}^k$ and subanalytic bumb function\footnote{Such functions exist. Of course they are not $\mathscr{C}^{\infty}$.}, this will result in a $\mathscr{C}^k$ and subanalytic perturbation of $C$.
\end{remark}
Given $\epsilon>0$ let 
\be\label{eq:fiepsilon}\varphi_{\epsilon}:=\epsilon\rho\cdot \varphi :\C\to \C\ee
and define $j_{\epsilon}:D\to \CP^2$ by:
\be j_{\epsilon}(x):=\phi^{-1}\left(\phi(j_{0}(x)), \varphi_{\epsilon}(\phi(j_{0}(x)))\right).\ee
Notice already that a map $j_{\epsilon}$ defined in this way, with $\varphi$ smooth, verifies the  properties (1) and (5) from Proposition \ref{propo:disk}. 

\subsubsection{The construction of the map $\varphi$} \label{sec:map}


The construction depends on the given sequence $\{a_d\}_{d\geq 1}.$ In the process we will also construct the sequence $\{d_m\}_{m\geq 1}$ of the degrees and consequently the subsequence $\{a_{d_m}\}_{m\geq 1}.$ We will search for a map $\varphi$ given by a series of polynomials, $\varphi=\sum_{m\geq 1}Q_m$, converging on $B$. While constructing the polynomials $\{Q_m\}_{m\geq 1}$ we will need auxiliary polynomials $\{q_m\}_{m\geq 1}$.

Let $d_1=1$ and $q_1(w)=Q_1(w)$ be the zero polynomial. 


Proceeding iteratively, for $m\geq 2$ let $q_{m}(w)$ be a polynomial in one complex variable with precisely $a_{d_{m-1}}$ zeroes on $B_{\frac{1}{2}}$, all of which are simple.  Define:
\be Q_{m}(w):=c_{m}(w-\lambda)^{2d_{m-1}}q_{m}(w),\ee
where $\lambda\notin B$ and  $c_{m}>0$ is some positive constant that we will fix later\footnote{The factor $(w-\lambda)^{2d_{m-1}}$ has two purposes: the power $2d_{m-1}$ makes sure that there are no cancellations in \eqref{eq:polymeps} and that $\deg(P_{m, \epsilon})=d_m)$ and   $\lambda$ is chosen to be outside $B$ to make sure that we do not introduce extra zeroes in $B$.}.
Define $d_{m}:=\deg(Q_{m}).$ Since $Q_{m}$ has only nondegenerate zeroes on $B_\frac{1}{2}$, there exists $\delta_{m}>0$ such that for every function $h\in \mathscr{C}^1(B,\C)$ with $\|h\|_{\mathscr{C}^1(B, \C)}\leq \delta_{m}$ we have
\be\label{eq:delta} \#Z(Q_{m})\cap B=\#Z(Q_{m}+h)\cap B,\ee
where all zeroes are still nondegenerate.

We choose now the sequence $\{c_m\}_{m\geq 1}$ such that for every $m\geq 1$
\be\label{eq:deltaseries} \|Q_{m}\|_{\mathscr{C}^1(B, \C)}< \frac{1}{2^{m}}\min\{\delta_1, \ldots, \delta_{m-1}\},\ee
and such that  the series $\sum_{m\geq 1}Q_m$ of holomorphic functions converges on $B$. We finally define:
\be \varphi(w):=\sum_{m\geq 1}Q_m(w).\ee
Notice that we have also obtained the sequences $\{d_m\}_{m\geq 1}$ and $\{a_{d_m}\}_{m\geq 1}.$
\subsubsection{The construction of the curves $\widetilde{Z}_{d_m}$.}\label{construction of the hypersurfaces} The construction from this section is inspired by \cite{GKP, BLN}.

 We still need to construct the sequence of curves $\{\widetilde{Z}_{d_m}\}_{m\geq 1}$. Each curve $\widetilde{Z}_{d_m}$ will be the zero set of a homogeneous polynomial $\widetilde{p}_{m}\in \C[z_0, z_1, z_2]_{(d_m)}$. We will first construct some auxiliary families of polynomials, letting the definition of $\widetilde{p}_m$ be the final step of this process.

For $m\geq 1$ and $\epsilon>0$ define the polynomial
\be\label{eq:polymeps} P_{m, \epsilon}(w_1, w_2):=w_2-\epsilon\sum_{j=1}^mQ_{j}(w_1).\ee
Notice that $\deg(P_{m, \epsilon})=d_m$. Moreover, on $B\times \C$ the system of equations
\be\label{eq:system1}\{w_2-\epsilon\varphi(w_1)=0, \, P_{m, \epsilon}(w_1, w_2)=0\}\ee
is equivalent to
\be \label{eq:system2}\left\{w_2-\epsilon \varphi(w_1)=0,\,Q_{m+1}(w_1)+\sum_{j\geq m+2}Q_{j}(w_1)=0\right\}.\ee
In fact, subtracting the second equation in \eqref{eq:system1} from the first equation in \eqref{eq:system1}, we get the second equation in \eqref{eq:system2}.
The solutions of \eqref{eq:system2} with $(w_1, w_2)\in B\times \C$ are all on the graph of 
\be\label{eq:fiepsilon2} \varphi_{\epsilon}:=\epsilon \varphi\ee (moreover, they are all positively oriented, since  $\varphi_{\epsilon}|_{B}$ is holomorphic). 
In other words, every solution to this system is of the form $(u, \varphi_{\epsilon}(u))$ with $u\in Z(Q_{m+1}+\sum_{j\geq m+2}Q_{j}).$ By \eqref{eq:deltaseries} it follows that
\be \left\|\sum_{j\geq m+2}Q_{j}\right\|_{{\mathscr{C}^1(B, \C)}}\leq \sum_{j\geq m+2}\frac{1}{2^j}\min\{\delta_1, \ldots, \delta_{j-1}\}<\delta_{m+1}.\ee In particular, from \eqref{eq:delta} the number of solutions of $\eqref{eq:system1}$ with $(w_1, w_2)\in B\times \C$ is the same as $\#Z(Q_{m+1})\cap B$, which is precisely $a_{d_{m}}$. Moreover, all these solutions are nondegenerate.

Let now $\widetilde{p}_m(z_0, z_1, z_2)$ be the homogenization of $P_{m, \epsilon}(z_1, z_2)$ and set $\widetilde{Z}_{d_m}:=Z(\widetilde{p}_m)$. Consider $Z(\widetilde{p}_m)\cap j_{\epsilon}(D)$. Notice that $\deg(\widetilde{p}_m)=d_m$. Let now 


\be \Gamma_{\epsilon}:=\mathrm{graph}(\varphi_{\epsilon}|_{B_{\frac{1}{2}}})\subset\phi^{-1}( j_{\epsilon}(D))\ee
and observe that  $\phi$ maps the set  $Z(\widetilde{p}_m)\cap \phi^{-1}(\Gamma_{\epsilon})$ to the set of solutions of \eqref{eq:system1}. In particular, since $\phi$ is a biholomorphism, by construction,
\be \label{eq:count}\#\left(Z(\widetilde{p}_m)\cap \phi^{-1}(\Gamma_{\epsilon})\right)=a_{d_m},\ee
with the intersection $Z(\widetilde{p}_m)\cap \phi^{-1}(\Gamma_{\epsilon})$ which is transversal.

\subsubsection{End of the proof of Proposition \ref{propo:disk}}
Property (2) follows from \eqref{eq:bump1} and the fact that $\rho \varphi_{\epsilon}|_{D_{\frac{1}{2}}}$ is holomorphic. Properties (4) and (5) follows from \eqref{eq:count} and the definition of the polynomials $\widetilde{p}_m$, whose zero sets are the $\widetilde{Z}_{d_m}.$ Property (3) follows from the fact that $ \rho \varphi_{\epsilon}\to 0$ in the $\mathscr{C}^1$--topology.
This concludes the proof.\qed

\section{Proof of Theorem \ref{thm:CP1}: the general case}\label{Section Proof theorem 1 general}
We continue now with the proof of Theorem \ref{thm:CP1} in the case of a general plane algebraic curve $C\hookrightarrow\CP^2$. The following proof was suggested to the authors by the anonymous referee.

Given the sequence $\{a_d\}_{d\in \mathbb{N}}$ and $\epsilon>0$, it will be enough to build a surface $C_\epsilon$, $\epsilon$--close to $C$ in the $\mathscr{C}^{\infty}$--topology, and infinitely many algebraic curves $\{Z_m\}_{m\in \mathbb{N}}$ of degree $\ell(m)$ such that $\#Z_m\cap C_\epsilon\geq a_{\ell(m)}$ (with transversal intersection). 

 Let $k\in \mathbb{N}$ be the degree of $C$. Choose affine coordinates $(w_1,w_2)$  on $\C^2\simeq U\subset \CP^2$ in such a way that the projection of $C$ to the $w_1$-axis is unramified over the disc $D =\{|w_1| \leq 1\}$ and of degree $k$. 
 
Given now the sequence $\{a_d\}_{d\in \mathbb{N}}$ we define the new sequence:
\be b_{d}:=a_{kd}\ee
 and we apply the construction of the previous section to this new sequence.
Let therefore $\varphi_\epsilon:D\to \C$ be the map defined in \eqref{eq:fiepsilon}, where $\varphi$ is constructed in Section \ref{sec:map}. Consider the smooth surface
$C_\epsilon\subset \CP^2$ obtained by perturbing $C$ over $D$, defined in affine coordinates by
\be C_{\epsilon}\cap U=\{R(w_1, w_2-\varphi_\epsilon(w_1))=0\}\ee
(recall that $\varphi_\epsilon$ is zero outside $D$).

Let $\widetilde{R}_{m, \epsilon}$ be the sequence of polynomials given by:
\be \widetilde{R}_{m, \epsilon}(w_1, w_2)=R(w_1, P_{m,\epsilon}(w_1, w_2)),\ee
where $P_{m, \epsilon}$ is defined in \eqref{eq:polymeps}.

By construction, $\widetilde{R}_{m,\epsilon}$ is a polynomial of degree $kd_{m}$ and the curves $\widetilde{Z}_{m}=Z(\widetilde{R}_{m, \epsilon})\subset \C^2$ intersect $C_\epsilon\cap U$ in at least $k b_{d_m}\geq a_{kd_m}$ points. In particular this is true also for the projective curve $Z_m=Z(R_{m, \epsilon})$, of degree $kd_m$, where $R_{m, \epsilon}$ is obtained by homogenization. The condition on the transversal intersection is obtained by a slight perturbation of $Z_m$.

This concludes the proof.
\begin{remark}\label{smoothness}
It is worth noting that the same proof works without requiring the curve $C$ to be smooth. The perturbation can be done in a small neighborhood $U$ of any smooth point $p\in C$ and the surface $C_\epsilon$ obtained is a $\mathscr{C}^\infty$ modification of $C$ around $p$ and coincides with $C$ outside $U$.
\end{remark}

\section{Proof of Theorem \ref{thm:surfaces}}\label{sec:proj}


Let $L$ be the very ample line bundle on $X$ associated with the very ample divisor $H$. For the generic choice of three holomorphic sections $s_0, s_1, s_2$ of $L$ the map $u:X\to \CP^2$ defined by
\be u(x):=[s_0(x), s_1(x), s_2(x)]\ee
satisfies:
\begin{enumerate}
\item $u$ is a branched covering of some degree $k\geq 1$;
\item $u|_C$ is a branched covering of the same degree $k$ on some projective curve $C':=u(C)\subset \CP^2$.
\end{enumerate}
Moreover, by construction, $u^*\mathcal{O}_{\CP^2}(1)=L$.

Let $p\in C'$ be a smooth point such that $u^{-1}(p)$ consists of $k$ points in $C$, and let  $D$ be a disk in $\CP^2$ around $p$ such that $u^{-1}(D)$ consists of $k$ disjoint disks in $X$. 

 Apply Theorem \ref{thm:CP1} (see also Remark \ref{smoothness}), to the sequence $\{a_d\}$ to get a surface $C_\epsilon'\subset \CP^2$ which is a small $\mathscr{C}^\infty$ modification of $C'$ near $p$ (in particular $C'_{\epsilon}$ is smooth near $p$, and it coincides with $C'$ outside a disk) and a sequence of  smooth curves $Z'_{d_m}\subset \CP^2$ of degree $d_m$ such that $\#C_\epsilon'\cap Z'_{d_m}\geq a_{d_m}.$
 
Finally, we set $C_\epsilon:=u^{-1}(C_\epsilon')$ (this is a smooth surface in $X$ which is $\epsilon$-close to $C$),  and $Z_{d_m}:=u^{-1}(Z'_{d_m})$, which are smooth algebraic curves. By construction we have $[Z_{d_m}]=d_m[H]$ and $\#C_\epsilon\cap Z_{d_m}\geq k a_{d_m}$.
This concludes the proof.

	\bibliographystyle{alpha}
	\bibliography{bez}
\end{document}